\documentclass[12pt,leqno,draft]{article} 
\usepackage{amsmath,amssymb,amsthm,amsfonts}
\usepackage{enumerate,color}
\makeatletter
\def\@cite#1#2{[{{\bfseries #1}\if@tempswa , #2\fi}]}
\renewcommand{\section}{%
\@startsection{section}{1}{\z@}
{0.5truecm plus -1ex minus -.2ex}%
{1.0ex plus .2ex}{\bfseries\large}}
\renewcommand{\subsection}{%
\@startsection{subsection}{1}{\z@}
{0.5truecm plus -1ex minus -.2ex}%
{1.0ex plus .2ex}{\bfseries\fontsize{13pt}{0pt}\selectfont}}
\def\@seccntformat#1{\csname the#1\endcsname.\ }
\makeatother

\numberwithin{equation}{section} 
\newtheorem{thm}{Theorem}[section]

\newtheorem{lem}[thm]{Lemma}
\newtheorem{proposition}[thm]{Proposition}
\theoremstyle{definition}

\newtheorem{remark}{Remark}[section]

\newcommand{\ep}{\varepsilon}
\newcommand{\pa}{\partial}
\newcommand{\RN}{\mathbb{R}^N}

\newcommand{\tmax}{T_{\rm max}}

\newcommand{\io}{\int_\Omega}

\newcommand{\Lpnorm}[2]{\|#1\|_{L^{#2}(\Omega)}}

\begin{document}
\footnote[0]
    {2010{\it Mathematics Subject Classification}\/. 
    Primary: 35K65; Secondary: 35B44, 92C17.}
\begin{center}
    \Large{{\bf 
Finite time blow-up in a parabolic--elliptic Keller--Segel system 
with nonlinear diffusion and signal-dependent sensitivity
          }}
\end{center}
\vspace{4pt}
\begin{center}
    Takahiro Hashira
    \\[1mm]
    {\small Kokugakuin high school \\
    2-2-3, Jingu-mae, Shibuya-ku, Tokyo 150-0001, Japan}\\
    {\tt t.hashira1220@gmail.com}\\
\end{center}

\vspace{2pt}

\begin{center}    
    \small \today
\end{center}

\vspace{2pt}
\newenvironment{summary}
{\vspace{.5\baselineskip}\begin{list}{}{%
     \setlength{\baselineskip}{0.85\baselineskip}
     \setlength{\topsep}{0pt}
     \setlength{\leftmargin}{12mm}
     \setlength{\rightmargin}{12mm}
     \setlength{\listparindent}{0mm}
     \setlength{\itemindent}{\listparindent}
     \setlength{\parsep}{0pt}
     \item\relax}}{\end{list}\vspace{.5\baselineskip}}
\begin{summary}
{\footnotesize {\bf Abstract.}
This paper is concerned with the parabolic--elliptic Keller--Segel system 
with nonlinear diffusion and signal-dependent sensitivity
\begin{align}\tag{KS}\label{system}
\begin{cases}
u_t=\Delta(u+1)^m-\nabla\cdot(u\chi(v)\nabla v),\quad &x\in\Omega, t>0,\\
0=\Delta v-v+u, &x\in\Omega, t>0
\end{cases}
\end{align}
under homogeneous Newmann boundary conditions and initial conditions, 
where $\Omega=B_R(0)\subset\RN$ ($N\geq3,\ R>0$) is a ball, 
$m\geq 1$, $\chi$ is a function satisfying that 
$\chi(s)\geq\chi_0(a+s)^{-k}$ ($k>0$, $\chi_0>0$, $a\geq 0$) for all $s>0$
and some conditions.
If the case that $m=1$ and $\chi(s)=\chi_0s^{-k}$, 
Nagai--Senba \cite{NSblow} established finite-time blow-up of 
solutions under the smallness conditions on a moment of initial data $u(x, 0)$ and some condition for $k\in(0,1)$. 
Moreover, if the case that $\chi(s)\equiv(\mbox{const.})$, 
Sugiyama \cite{S2006} showed finite-time blow-up of solutions under the condition $m\in[1,2-\frac{2}{N})$. 
According to two previous works, it seems that the smallness conditions of $m$ and $k$ leads to finite-time blow-up of solutions.
The purpose of this paper is to give the relationship which depends only on $m$, $k$ and $N$ 
such that there exists initial data which corresponds finite-time blow-up solutions.
}
\end{summary}
\vspace{10pt}


\newpage 
\section{Introduction}
The chemotaxis system introduced by Keller and Segel in \cite{KS} 
describes a part of the life cycle of cellular slime molds with the chemotaxis.
More precisely, when cellular slime molds plunge into hunger,
they move towards higher concentrations of the chemical signal substance 
secreted by cells. 
After that, the system has been generalized to a nonlinear diffusion system which has been suggested 
in the survey by Hillen--Painter \cite{HP}. 
The system which is called Keller--Segel system and describes the aggregation of the species by chemotaxis
\begin{align*}
&u_t=\nabla\cdot(D(u)\nabla u-S(u)\chi(v)\nabla v),\\
&\tau v_t=\Delta v-v+u,
\end{align*}
where $D, S, \chi$ are some functions and $\tau=0,1$.
In this context $u(x,t)$ represents the density of the cell and $v(x,t)$ denotes 
the density of the semiochemical at place $x\in\RN$ and time $t>0$. 

From a biological view it is a meaningful question whether or not cells aggregate and 
how cells behave if they aggregate. 
In particular, the aggregation is mathematically defined by finite time blow-up, 
i.e., $\limsup_{t\to T}\Lpnorm{u(t)}{\infty}=\infty$ 
for some $T\in(0,\infty)$. 
At first, we recall some known results related to the following quasilinear system; 
such that this problem when $\chi(v)\equiv 1$:
\begin{align*}
&u_t=\nabla\cdot(D(u)\nabla u-S(u)\nabla v),\\
&\tau v_t=\Delta v-v+u.
\end{align*}
When $D(u)=m(u+\ep)^{m-1}$ and $S(u)=u(u+\ep)^{q-2}$ with $m\geq1, q\geq2, \ep\geq0$, 
if $q<m+\frac{2}{N}$, global existence was established by Sugiyama--Kunii \cite{SugiKuni}, Ishida--Yokota \cite{IY2,IYremark} and 
Ishida--Seki--Yokota \cite{ISY}; 
if $q=m+\frac{2}{N}$, Ishida--Yokota \cite{IYsmall,IYremark} (including the case $q\geq m+\frac{2}{N}$) and 
Blanchet--Lauren\c{c}ot \cite{Blan} showed that 
there exist global solutions with small initial data, and 
Lauren\c{c}ot--Mizoguchi \cite{mizo} showed that there exist large radially symmetric initial data 
such that the corresponding solutions blow up in finite time; 
if $q>m+\frac{2}{N}$, 
Sugiyama \cite{S2006} and Hashira--Ishida--Yokota \cite{HIY} showed existence of initial data 
such that the corresponding solutions blow up in finite time.
Moreover, when $D$ and $S$ are general, Tao--Winkler \cite{Tao} and Ishida--Seki--Yokota \cite{ISY} proved that 
solutions remain bounded under condition that $S(u)/D(u)\leq K(u+\ep)^\alpha$ for $u>0$ 
with $\alpha<\frac{2}{N}$; 
Cie\'slak--Stinner \cite{JDE} showed existence of initial data 
such that the corresponding solutions blow up in finite time 
with $S(u)/D(u)\geq Ku^{\frac{2}{N}+\delta}$ for $u>1$ with $K>0$ and $\delta>0$.

Next, we introduce some known results related to the following Keller--Segel system with 
signal-dependent sensitivity: 
\begin{align*}
&u_t=\Delta u-\nabla\cdot(u\chi(v)\nabla v),\\
&\tau v_t=\Delta v-v+u,
\end{align*}
when $D(u)=1, S(u)=u$.
In the case that $\chi$ is a singular sensitivity, that is, 
$\chi(v)=\frac{\chi_0}{v}$, 
Winkler \cite{wink} attained global existence of classical solutions 
when $\chi_0<\sqrt{\frac{2}{N}}$ and global existence of weak solutions when $\chi_0<\sqrt{\frac{N+2}{3N-4}}$. 
After that, Fujie \cite{Fujie} 
proved global existence and boundedness of classical solutions when 
$\chi(v)=\frac{\chi_0}{v}$ and $\chi_0<\sqrt{\frac{2}{N}}$.
In the case that $\chi$ satisfies that $\chi(v)\leq\frac{\chi_0}{(a+v)^k}$ $(k\geq1, a\geq0)$, 
Mizukami--Yokota \cite{MY} filled in the gap between singular type and regular type, i.e.,
they showed that global existence and boundedness under condition 
$\chi_0<k(a+\eta)^{k-1}\sqrt{\frac{2}{N}}$. 
Recently, Ahn \cite{Ahn} improved the smallness condition for $\chi_0$, 
and showed global existence and boundedness.
On the other hands, if $\tau=0$ and $\chi(v)=\frac{\chi_0}{v}$ with $\chi_0>\frac{2N}{N-2}$ and $N\geq3$, 
Nagai--Senba \cite{NSblow} proved that there exists some initial data $u_0$ such that radial solution blows up in finite time; moreover, 
in the case $\chi(v)=\frac{\chi_0}{v^k}$ ($k\in(0,1)$ with $N\geq3$), 
there exists some initial data $u_0$ such that radial solution blows up in finite time in \cite{NSblow}.

The purpose of this paper is to reserch the behavior of solutions to the mixed type model, i.e., 
to give the condition which derives finite-time blow-up solutions 
to Keller--Segel system with nonlinear diffusion and 
signal-dependent sensitivity.
Throughout this paper we consider the following 
parabolic--elliptic Keller--Segel system with nonlinear diffusion $D(u)=m(u+1)^{m-1}$ 
and signal-dependent sensivitity:
\begin{align}\label{KS}
    \begin{cases}
\dfrac{\pa u}{\pa t}=\Delta (u+1)^m -\nabla\cdot\left(u\chi(v)\nabla v\right), &x\in \Omega,\ t>0,\\[2.5mm]
    		   0=\Delta v-v+u,&x\in \Omega,\ t>0,\\[2.5mm]
    		\dfrac{\pa u}{\pa \nu}=\dfrac{\pa v}{\pa \nu}=0,&x\in \pa\Omega,\ t>0,\\[2mm]
                u(x,0)=u_0(x),\ &x\in \Omega,
    \end{cases}
\end{align}
where $\Omega=B_R(0):=\{x\in\RN : |x|<R\}\ (N\geq 3)$ is a ball with some $R>0$, 
$m\geq1$, $\frac{\pa}{\pa\nu}$ denotes differentiation with respect to outward normal of $\pa\Omega$, 
the initial data are supposed to satisfy 
\begin{align}\label{ini}
u_0 \in C^0(\overline{\Omega}\,)\backslash\,\{0\}\ \mbox{is radially symmetric and nonnegative}.
\end{align}
To attain the goal of this paper we will suppose that $\chi$ satisfies the following conditions:
\begin{align}\label{chi}
\chi(s)\geq\dfrac{\chi_0}{(a+s)^k}\quad\ \mbox{with}\ \chi_0>0,\ a\geq0\ \mbox{and}\ k>0\ \mbox{for all}\ s>0,
\end{align}
and 
\begin{align}\label{chi4}
\chi\in C^1((0,\infty))\cap L^{\infty}(\eta,\infty),
\end{align}
where $\eta$ is defined in \eqref{etadef} below.
%
%
%
%

Now we state the main theorem. 

\begin{thm}\label{thm}\label{thm1}
Let $\Omega=B_R(0)\subset\RN$ with $N\geq3$ and $R>0$ 
and let $m\geq1$, $M_0>0$, $M_1\in(0, M_0)$ and $L>0$. 
Assume that $m$ and $\chi$ satisfy \eqref{chi} and \eqref{chi4} with
\begin{align}\label{condi}
1\leq m<2-\dfrac{2}{N}\ \ \mbox{and}\ \ 
0<k<\min\left\{ \dfrac{2}{N-2},\, \dfrac{N-2-N(m-1)}{\left((m-1)N+1\right)(N-2)} \right\}.
\end{align}
Then, there exist constants $\ep_0>0$, $r_1\in(0, R)$ and $T^{*}>0$ the following property\,$:$ 

If $u_0$ with \eqref{ini} satisfies 
\begin{align*}
\io u_0(x)\,dx=M_0, \ \ \mbox{and}\ \ \int_{B_{r_1}(0)}u_0(x)\, dx\geq M_1
\end{align*}
as well as
\begin{align*}
u_0(x)\leq L|x|^{-\frac{N(N-1)}{(m-1)N+1}-\ep_0}\quad\mbox{for all}\ x\in\Omega,
\end{align*}
then the solution $(u, v)$ to \eqref{KS} blows up at $t=T^{*}<\infty$ in the sense that
\begin{align}\label{ulim}
\limsup_{t\nearrow T^{*}}\Lpnorm{u(\cdot, t)}{\infty}=\infty.
\end{align}
\end{thm}

\smallskip

\begin{remark}
If $m=1$, $N=3, 4$ and $\chi(v)=\frac{\chi_0}{v^k}$, then the condition for $k$ connects with that in \cite{NSblow}; 
indeed, in the case $m=1$, since
\begin{align*}
\min\left\{ \dfrac{2}{N-2},\, \dfrac{N-2-N(m-1)}{\left((m-1)N+1\right)(N-2)} \right\}
=\min\left\{ \dfrac{2}{N-2},\, 1 \right\}=1,
\end{align*}
the conditions \eqref{condi} are reduced to that in \cite{NSblow}. 
Thus Theorem \ref{thm1} is a generalization of the previous work 
in the case $N=3, 4$.
\end{remark}

Our technical approach to the proof of Theorem \ref{thm} is to comprise 
the following ordinary differential inequality of a moment type function $\phi$\,: 
\begin{align}\tag{A}\label{A}
\dfrac{d\phi}{dt}(t)\geq C_1\phi^2(t)-C_2\quad\mbox{for all}\ t\in(0,\tmax)
\end{align}
with some constant $C_1, C_2>0$, where $\phi$ is defined as
\begin{align*}
\phi(t)=\int_0^{s_0}s^{-\gamma}(s_0-s)w(s,t)\,ds,\quad \mbox{for all}\ \ t\in[0,\tmax)
\end{align*}
with some $\gamma\in(0,1)$ and $s_0\in(0, R^N)$ with 
\begin{align*}
w(s,t)=\int_0^{s^\frac{1}{N}}\rho^{N-1}u(\rho,t)\,d\rho,\quad 
\mbox{for all}\ \  s\in[0,R^N]\ \ \mbox{and}\ \  t\in(0,\tmax).
\end{align*}
By a straightforward ODI comparison we can obtain finite-time blow-up for the moment type function if initial data $u_0$ 
satitsfies $\int_{B_{r_1}(0)}u_0\geq M_1$.
This idea based on a recent method, for example \cite{BFL, tana, wink3}. 
In Section \ref{2} we recall some results, for instance, 
the $L^\infty$-estimate for $u$ and $v$. 
Moreover, we give the proof of key estimate \eqref{A} 
in Section \ref{3} which is the main part of this paper.
Finally we complete the proof of Theorem \ref{thm} in Section \ref{4}.

%
%
\section{Preliminaries}\label{2}

In the section we state some known results concerning local existence of classical solutions to \eqref{KS} 
and some helpful properties of classical solutions.

\begin{lem}[{{\it Local existence of classical solutions}}]\label{regu}
Let $N\geq1$, $R>0$, $m\geq1$ and $M_0>0$.
Assume that 
$\chi$ satisfies \eqref{chi} and \eqref{chi4}, and $u_0$ satisfies 
\eqref{ini} and $\io u_0=M_0$.
Then there exists $\tmax\in(0,\infty]$ and an exactly pair $(u, v)$ of 
radially symmetric and nonnegative functions
\begin{align*}
(u,v)\in C^0([0, \tmax);C^0(\Omega))\cap 
C^{2,1}(\overline{\Omega}\times(0, \tmax))\times
C^{2,0}(\overline{\Omega}\times(0,\tmax)),
\end{align*}
which solves \eqref{KS} classically. Moreover, 
\begin{equation*}
\mbox{either}\ \ \tmax=\infty,\ \mbox{or}\ \ \tmax<\infty\ \ \mbox{with}
\ \ {\limsup_{t \nearrow \tmax}} \Lpnorm{u(\cdot,t)}{\infty}=\infty,
\end{equation*}
is fulfilled, and
\begin{equation}\label{umas}
\io u(x,t)\,dx=\io v(x,t)\, dx=M_0\quad\mbox{for all}\ t\in(0,\tmax)
\end{equation}
as well as
\begin{equation}\label{vesti}
\inf_{x\in\Omega}v(x,t)\geq\eta\quad\mbox{for all}\ t\in(0, \tmax)
\end{equation}
with
\begin{equation}\label{etadef}
\eta:=\Lpnorm{u_0}{1}\int_0^\infty\dfrac{1}{(4\pi t)^\frac{N}{2}}e^{-\left(t+\frac{({\rm diam}\,\Omega)^2}{4t}\right)}\,dt>0
\end{equation}
holds. 
\end{lem}

\begin{proof}
The statements concerning existence, regularity and extensibility are well-known, for example, 
refer to \cite[Lemma 1]{Vig}.
The identity \eqref{umas} directly follows from integration of the first equation in \eqref{KS} in $\Omega$ 
and Neumann boundary condition, as well as integration the second equation 
\begin{align*}
\io u(x,t)\,dx=\io v(x,t)\,dx\quad\mbox{for all}\ t\in(0,\tmax).
\end{align*}
Moreover, the inequality \eqref{vesti} is given the same way as in the proof of \cite[Lemma 2.1]{FWY}.
\end{proof}

The following pointwise estimate for $u$ will play an important role in the proof of finite time blow-up. 
We can obtain the $L^\infty$-estimate of $u$ by using \cite[Theorem1.1]{Mario}.

\begin{lem}\label{uesti}
Let $\Omega=B_R(0)\subset\RN$ \ $(N\geq3$, $R>0)$, $M_0>0$ and $L>0$, 
and assume that $\chi$ satisfies \eqref{chi} and \eqref{chi4} with $k\in(0,1)$. 
Assume that $m$ fulfill that
\begin{equation*}
m\in\left(1-\dfrac{1}{N}, 1+\dfrac{N-2}{N}\right].
\end{equation*}
For all $\ep>0$ set $p(\ep):=\frac{N(N-1)}{(m-1)N+1}+\ep$.
Then there exists $K>0$ such that the following property\,$:$

If $u_0$ satisfies \eqref{ini} and $\io u_0=M_0$ as well as
\begin{equation*}
u_0(r)\leq Lr^{-p(\ep)}\quad \mbox{for all}\ r\in(0,R),
\end{equation*}
and $(u, v)$ is a classical solution to \eqref{KS}, 
then 
\begin{equation*}
u(r,t)\leq Kr^{-p(\ep)}\quad\mbox{for all}\ r\in(0,R)\ \mbox{and}\ t\in(0,\tmax).
\end{equation*}
\end{lem}

\begin{proof}
At first, using \eqref{vesti}, we have $v(x,t)\geq \eta$ for all $(x,t)\in\Omega\times(0,\tmax)$. 
Therefore, by \eqref{chi4}, we can find that 
\begin{align}\label{chi3}
\chi(v(x,t))\leq\|\chi\|_{L^{\infty}(\eta,\infty)}\quad\ \mbox{for all}\ (x,t)\in\Omega\times(0,\tmax)
\end{align}
and $\int_\eta^{v(x,t)} \chi(s)\,ds\geq0$ for all $(x,t)\in\Omega\times(0,\tmax)$.
Next, we put $V(x,t):=\int_\eta^{v(x,t)} \chi(s)\,ds$, 
$D(x,t,\rho):=m(\rho+1)^{m-1}$ and 
$S(x,t,\rho):=-\rho$. 
Then we obtain from \eqref{KS} that
\begin{align*}
\begin{cases}
u_t\leq\nabla\cdot(D(x,t,u)\nabla u+S(x,t,u)\nabla V)\quad&\mbox{in}\ \Omega\times(0,\tmax),\\
(D(x,t,u)\nabla u+S(x,t,u)\nabla V)\cdot\nu=0,&\mbox{on}\ \pa\Omega\times(0,\tmax),\\
u(\cdot,0)=u_0,&\mbox{in}\ \Omega,
\end{cases}
\end{align*}
where $\nu$ is the outward normal vector to $\pa\Omega$.
We have clearly 
\begin{align*}
&D(x,t,\rho)\geq m\rho^{m-1},\\
&D(x,t,\rho)\leq 2^{m-1}m\max\{\rho,1\}^{m-1},\\
&|S(x,t,\rho)|\leq \max\{\rho,1\}
\end{align*}
for all $x\in\Omega$, $t\in(0,\tmax)$ and $\rho\in(0,\infty)$, and
\begin{equation*}
\io u(\cdot, 0)=M_0.
\end{equation*}
Now we can pick large enough $\theta>N$ satisfying that 
\begin{equation*}
m-1\in\left(\dfrac{1}{\theta}-\dfrac{1}{N}, \dfrac{1}{\theta}+\dfrac{N-2}{N}\right]
\end{equation*}
as well as
\begin{equation*}
p(\ep)=\dfrac{N(N-1)}{(m-1)N+1}+\ep>\dfrac{N(N-1)}{(m-1)N+1-\frac{N}{\theta}}=\dfrac{N-1}{m-1+\frac{1}{N}-\frac{1}{\theta}}.
\end{equation*}
Next the second equation imply that 
\begin{equation}\label{second}
-\dfrac{1}{r^{N-1}}\left(r^{N-1}v_r(r,t)\right)_r=u(r,t)-v(r,t)
\end{equation}
for all $r\in(0,R)$ and $t\in(0,\tmax)$.
Using \eqref{second} and \eqref{umas}, we obtain that
\begin{align*}
\left|r^{N-1}v_r(r,t)\right|&=\left|\int_0^r\left(\rho^{N-1}v_r(\rho,t)\right)_r\,d\rho\right|\\
               &=\left|\int_0^r\rho^{N-1}(v(\rho,t)-u(\rho,t))\,d\rho\right|\\
               &\leq\int_0^R\rho^{N-1}(v(\rho,t)+u(\rho,t))\,d\rho=\dfrac{2M_0}{\omega_N}
\end{align*}
for all $r\in(0,R)$ and $t\in(0,\tmax)$, 
where $\omega_{N-1}$ is the $(N-1)$-dimensional measure of the sphere $\pa B_1(0)$.
According to the above estimate, \eqref{chi3} and \eqref{vesti}, we see that
\begin{align*}
\io|x|^{(N-1)\theta}|\nabla V(x,t)|^{\theta}\,dx
&=\io|x|^{(N-1)\theta}\left|\chi(v(x,t))\nabla v(x,t)\right|^\theta\,dx\\
&\leq\dfrac{\|\chi\|_{L^{\infty}(\eta,\infty)}^\theta}{\omega_{N-1}}\int_0^R r^{N-1}|r^{N-1}v_r(r,t)|^\theta\,dr\\
&\leq\dfrac{\|\chi\|_{L^{\infty}(\eta,\infty)}^\theta}{\omega_{N-1}}\cdot\left(\dfrac{2M_0}{\omega_{N-1}}\right)^\theta\int_0^Rr^{N-1}\,dr\leq K_{\nabla v}
\end{align*}
for all $t\in(0,\tmax)$, where $K_{\nabla v}:=\|\chi\|_{L^{\infty}(\eta,\infty)}^\theta\cdot\left(\frac{2M_0}{\omega_{N-1}}\right)^\theta\cdot|\Omega|$.
Thus, from \cite[Theorem 1.1]{Mario} the claim is proved.
\end{proof}

Next we introduce a uniform-in-time upper estimate of $v$ proved as in the proof of \cite[lemma 3.2]{SSD}. 

\begin{lem}
Let $N\geq3$ and $R>0$. 
Then for all $M_0>0$ there exists $C_v=C_v(R,M_0)>0$ such that 
if $u_0$ with \eqref{ini} and $\io u_0=M_0$, then
\begin{equation*}
|v_r(r,t)|\leq C_vr^{1-N}\quad\mbox{for all}\ r\in(0,R)\ \mbox{and}\ t\in(0,\tmax)
\end{equation*}
as well as
\begin{equation}\label{vesti2}
v(r,t)\leq C_vr^{2-N}\quad\mbox{for all}\ r\in(0,R)\ \mbox{and}\ t\in(0,\tmax)
\end{equation}
holds.
\end{lem}

%
\section{Differential inequality for a moment type function}\label{3}
In this section assume that $u_0$ satisfies \eqref{ini} and $\io u_0=M_0$ with $M_0>0$ 
we let 
$(u,v)=(u(r,t),v(r,t))$ (for $r=|x|,\ t>0$) denote the corresponding local solution of \eqref{KS}, 
and we set functions $w=w(s,t)$ and $z=z(s,t)$ given by 
\begin{equation*}
w(s,t):=\int_0^{s^\frac{1}{N}}\rho^{N-1}u(\rho,t)\,d\rho,\quad s\in[0, R^N],\ t\in[0,\tmax),
\end{equation*}
as well as
\begin{equation*}
z(s,t):=\int_0^{s^\frac{1}{N}}\rho^{N-1}v(\rho,t)\,d\rho,\quad s\in[0, R^N],\ t\in[0,\tmax).
\end{equation*}
Then
\begin{equation}\label{westi}
w_s(s,t)=\dfrac{1}{N}u(s^\frac{1}{N},t)\quad\mbox{and}\quad
w_{ss}(s,t)=\dfrac{1}{N^2}s^{\frac{1}{N}-1}u_r(s^\frac{1}{N},t),
\end{equation}
for all $s\in(0,R^N)$ and $t\in(0,\tmax),$ and similarly
\begin{equation}\label{zesti}
z_s(s,t)=\dfrac{1}{N}v(s^\frac{1}{N},t)\quad\mbox{and}\quad
z_{ss}(s,t)=\dfrac{1}{N^2}s^{\frac{1}{N}-1}v_r(s^\frac{1}{N},t),
\end{equation}
for all $s\in(0,R^N)$ and $t\in(0,\tmax)$ given directly by calculation.
According to the definition of $w$ and \eqref{umas}, 
we can find easily that
\begin{equation*}
0=w(0,t)\leq w(s,t)\leq w(R^N,t)=\int_0^R\rho^{N-1}u(\rho,t)\,d\rho =\dfrac{1}{\omega_N}\io u(x,t)\,dx=\dfrac{M_0}{\omega_N}
\end{equation*}
for all $s\in(0, R^N)$ and $t\in[0, \tmax)$.
Moreover, for $s_0\in(0,R^N)$, we introduce 	the moment type function $\phi(s_0,\,\cdot)$ : $[0, \tmax)\to{\mathbb R}$ defined by
\begin{equation}\label{phidef}
\phi(s_0,t):=\int_0^{s_0}s^{-\gamma}(s_0-s)w(s,t)\,ds,\quad \mbox{for all}\ \ t\in[0,\tmax)
\end{equation}
with some constants $\gamma\in(0,1)$ which belongs to 
$C^0([0,\tmax))\cap C^1((0,\tmax))$ as a function of variable t.

Our goal in this section is to give a proof for the following proposition.

\begin{proposition}\label{ODI}
Let $N\geq3$ and $R>0$. 
Assume that $m$ and $\chi$ satisfy \eqref{chi} and \eqref{chi4} with
\begin{equation*}
1\leq m<2-\dfrac{2}{N}\ \ \mbox{and}\ \ 
0<k<\min\left\{ \dfrac{2}{N-2},\, \dfrac{N-2-N(m-1)}{\left((m-1)N+1\right)(N-2)} \right\}.
\end{equation*}
Then, there exist $\ep_0>0$ and such that the following property\,$:$

For all $M_0>0$, $L>0$ and
\begin{equation}\label{propgamma}
\gamma\in\left(1-\dfrac{2}{N}-\dfrac{p(\ep_0)}{N}(m-1), \,\min\left\{2-\dfrac{4}{N}-\dfrac{2p(\ep_0)}{N}(m-1)-\left(1-\dfrac{2}{N}\right)k,\,1\right\}\right)
\end{equation}
with $p(\ep_0)=\frac{N(N-1)}{(m-1)N+1}+\ep_0$, one can find 
$s_1\in(0,R^N)$, $C>0$ and $\theta_1\in(0, 2-(1-\frac{2}{N})k)$ such that whenever $u_0$ satisfies \eqref{ini}, $\io u_0=M_0$ and
\begin{equation}\label{iniprop}
u_0(r)\leq Lr^{-p(\ep_0)}\quad\mbox{for all}\ \ r\in(0,R),
\end{equation}
the function $\phi$ defined in \eqref{phidef} satisfies
\begin{equation}\label{propmain}
\dfrac{\pa\phi}{\pa t}(s_0,t)\geq \dfrac{1}{C} s_0^{-3+\gamma+(1-\frac{2}{N})k}\phi^2(s_0,t)-Cs_0^{3-\gamma-\theta_1}
\end{equation}
for all $s_0\in(0,s_1)$ and $t\in(0, \tmax).$
\end{proposition}
The first step for giving the proof of Proposition \ref{ODI} 
is to rewrite \eqref{KS} to the equation of $w$ and $z$.

\begin{lem}\label{lem3.2}
Let $N\geq3$, $R>0$, $m\geq1$ and assume that $\chi$ satisfies \eqref{chi} and \eqref{chi4}, 
and assume that $u_0$ satisfies \eqref{ini}. 
Then for all $\gamma\in(0,1)$, $s_0\in(0,R^N)$ and $t\in(0,\tmax)$, 
\begin{align}\label{Idef}
\dfrac{\pa\phi}{\pa t}(s_0,t)
=&\,mN^2\int_0^{s_0}s^{2-\frac{2}{N}-\gamma}(s_0-s)(Nw_s+1)^{m-1}w_{ss}(s,t)\,ds\\\notag
&+N\int_0^{s_0}s^{-\gamma}(s_0-s)\chi(Nz_s(s,t))w(s,t)w_s(s,t)\,ds\\\notag
&-N\int_0^{s_0}s^{-\gamma}(s_0-s)\chi(Nz_s(s,t))z(s,t)w_s(s,t)\,ds\\\notag
=:&\,I_1(s_0,t)+I_2(s_0,t)+I_3(s_0,t).
\end{align}
\end{lem}

\begin{proof}
At first, we fix $u_0$ with \eqref{ini}. 
The radially symmetric solution $(u(r), v(r))=(u(r,\cdot), v(r,\cdot))$ with $r=|x|$ to \eqref{KS} satisfies
\begin{equation*}
u_t(r)=r^{1-N}\left(r^{N-1}\left((u(r)+1)^m\right)_r\right)_r
-r^{1-N}\left(r^{N-1}u(r)\chi(v(r))v_r(r)\right)_r
\end{equation*}
as well as
\begin{equation*}
-\dfrac{1}{r^{N-1}}\left(r^{N-1}v_r(r)\right)_r=u(r)-v(r)
\end{equation*}
for $r\in(0,R)$. 
Therefore, we have
\begin{align}\label{2nd}
r^{N-1}v_r(r)=\int_0^r\left(\rho^{N-1}v_r(\rho)\right)_r\,d\rho
=\int_0^r\rho^{N-1}(v(\rho)-u(\rho))\,d\rho
\end{align}
for all $r\in(0,R)$. 
Then, using \eqref{2nd}, \eqref{westi} and \eqref{vesti}, we obtain that
\begin{align}\label{wpara}
w_t(s)=&\,\int_0^{s^\frac{1}{N}}\rho^{N-1}u_t(\rho)\,d\rho\\\notag
=&\,\int_0^{s^\frac{1}{N}}\left(\rho^{N-1}\left((u(\rho)+1)^m\right)_r\right)_r\,d\rho
-\int_0^{s^\frac{1}{N}}\left(\rho^{N-1}u(\rho)\chi(v(\rho))v_r(\rho)\right)_r\,d\rho\\\notag
=&\,ms^{1-\frac{1}{N}}(u(s^\frac{1}{N})+1)^{m-1}u_r(s^\frac{1}{N})
-\,s^{1-\frac{1}{N}}\chi(v(s^\frac{1}{N}))u(s^\frac{1}{N})v_r(s^\frac{1}{N})\\\notag
=&\,mN^2s^{2-\frac{2}{N}}(Nw_s(s)+1)^{m-1}w_{ss}(s)
-N\chi(Nz_s(s))w_s(s)(z(s)-w(s))
\end{align}
for all $s\in(0,R^N)$.
Thanks to \eqref{phidef} and \eqref{wpara}, we obtain \eqref{Idef}.
\end{proof}

We next give the following estimate for $I_2$ 
by using the condition \eqref{chi} and \eqref{vesti2}.
Here we note that, in the following lemma, 
it is important thet we can pick any $s_0\in(0, R^N)$.

\begin{lem}\label{I2esti}
Let $N\geq3$, $\gamma\in(0,1)$.
Assume that $\chi$ satisfies \eqref{chi} and \eqref{chi4}, 
and that \eqref{vesti2} holds with $C_v>0$. 
Then there exists $C=C(R,\chi_0,N,k,C_v)>0$ such that for all $t\in(0,\tmax)$ and $s_0\in(0,R^N)$,
\begin{equation*}
I_2(s_0,t)\geq C\int_0^{s_0}s^{-\gamma+(1-\frac{2}{N})k}(s_0-s)w(s,t)w_s(s,t)\,ds
\end{equation*}
holds.
\end{lem}

\begin{proof}
Using \eqref{zesti}, \eqref{vesti2} and the fact $r^N=s$, we estimate
\begin{equation*}
Nz_s(s,t)=v(s^\frac{1}{N},t)\leq C_vs^\frac{2-N}{N}
\end{equation*}
for all $s\in(0,R^N)$ and $t\in(0,\tmax)$.
From \eqref{chi}, this estimate and $k>0$ 
we derive the inequality
\begin{align*}
\chi(Nz_s(s,t))&\geq\chi_0(a+Nz_s(s,t))^{-k}\\
&\geq\chi_0(a+C_vs^\frac{2-N}{N})^{-k}\\
&=\chi_0(as^\frac{N-2}{N}+C_v)^{-k}s^{(1-\frac{2}{N})k}\\
&\geq\chi_0(aR^{N-2}+C_v)^{-k}s^{(1-\frac{2}{N})k}\quad\mbox{for all}\ s\in(0,R^N)\ \mbox{and}\ t\in(0,\tmax).
\end{align*}
Therefore we can find that for all $s_0\in(0,R^N)$ and $t\in(0,\tmax)$,
\begin{align*}
I_2(s_0,t)\geq N\chi_0(aR^{N-2}+C_v)^{-k}\int_0^{s_0}s^{-\gamma+(1-\frac{2}{N})k}(s_0-s)w(s,t)w_s(s,t)\,ds
\end{align*}
due to $\chi_0>0$ and nonnegativity of $w$ and $w_s$.
\end{proof}

The next step is to give the estimate for $I_1$. 
To attain the purpose of this step we introduce two lemmas.
The following lemma has already been proved in \cite[Lemma 3.3]{BFL}.

\begin{lem}\label{betafunc}
For all $a>-1$ and $b>-1$ and $s_0\geq0$ we have
\begin{equation*}
\int_0^{s_0}s^a(s_0-s)^b\,ds=B(a+1,B+1)s_0^{a+b+1},
\end{equation*}
where $B$ is Eular's beta function.
\end{lem}

\begin{lem}\label{lem35}
Let $N\geq3$ and $k\in(0,\frac{N}{N-2})$, and asuume that $\gamma\in(0,1)$ satisfies
\begin{equation}\label{alphadef}
\alpha:=\gamma-\left(1-\dfrac{2}{N}\right)k\in(0,1).
\end{equation}
Then for any $s_0\in(0,R^N)$,
\begin{equation}\label{lem35main}
w(s,t)\leq\sqrt{2}s^\frac{\alpha}{2}(s_0-s)^{-\frac{1}{2}}\left\{\int_0^{s_0}s^{-\alpha}(s_0-s)w(s,t)w_s(s,t)\,ds\right\}^\frac{1}{2}
\end{equation}
for all $s\in(0,R^N)$ and $t\in(0,\tmax)$.
\end{lem}
\begin{proof}
By using the function 
\begin{align*}
\psi(s):=\dfrac{1}{2}s^{-\gamma+(a-\frac{2}{N})k}(s_0-s)w^2(s, t)\quad\mbox{for all}\quad(t\in(0, \tmax))
\end{align*}
instead of $\psi$ in the proof of \cite[Lemma 4.2]{wink3}, we can prove that \eqref{lem35main} holds.
\end{proof}

Next we show that there exists $\gamma\in(0,1)$ satisfying \eqref{propgamma}.

\begin{lem}\label{gammaexist}
Let $N\geq3$, and let $m\geq1$ and $k>0$ be such that
\begin{equation}\label{lem36}
1\leq m<2-\dfrac{2}{N}\ \ \mbox{and}\ \ 
0<k<\min\left\{ \dfrac{2}{N-2},\, \dfrac{N-2-N(m-1)}{\left((m-1)N+1\right)(N-2)} \right\}.
\end{equation}
Then there exists $\ep_0>0$ such that the set
\begin{equation*}
\left(1-\dfrac{2}{N}-\dfrac{p(\ep_0)}{N}(m-1), \,\min\left\{2-\dfrac{4}{N}-\dfrac{2p(\ep_0)}{N}(m-1)-\left(1-\dfrac{2}{N}\right)k,\,1 \right\}\right)\neq\emptyset
\end{equation*}
is a subset of $(0,1)$, here $p(\ep_0)=\frac{N(N-1)}{(m-1)N+1}+\ep_0$.
\end{lem}

\begin{proof}
From \eqref{lem36} and $N\geq3$, we can pick $\ep_0>0$ such that
\begin{equation*}
0<\ep_0<\dfrac{N-2}{m-1}\left(\dfrac{N-2-N(m-1)}{\left((m-1)N+1\right)(N-2)}-k\right),
\end{equation*}
when $m\in(1,2-\frac{2}{N})$.
Then due to $N\geq3$, it follows that 
\begin{align*}
&1-\dfrac{2}{N}-\dfrac{p(\ep_0)}{N}(m-1)-\left(1-\dfrac{2}{N}\right)k\\
&\quad=\dfrac{N-2}{(m-1)N}\left(\left(\dfrac{N-2-N(m-1)}{\left((m-1)N+1\right)(N-2)}-k\right)-\dfrac{m-1}{N-2}\ep_0\right)>0.
\end{align*}
On the other hand, $N\geq3$ and \eqref{lem36} lead to that
\begin{equation*}
1-\dfrac{2}{N}-\dfrac{p(\ep_0)}{N}(m-1)-\left(1-\dfrac{2}{N}\right)k=\left(1-\dfrac{2}{N}\right)(1-k)>0
\end{equation*}
for any $\ep_0>0$, when $m=1$, thus we can find that
\begin{equation*}
\left(1-\dfrac{2}{N}-\dfrac{p(\ep_0)}{N}(m-1), \,2-\dfrac{4}{N}-\dfrac{2p(\ep_0)}{N}(m-1)-\left(1-\dfrac{2}{N}\right)k\right)\neq\emptyset,
\end{equation*}
as well as 
\begin{equation*}
0<\left(1-\dfrac{2}{N}\right)k<1-\dfrac{2}{N}-\dfrac{p(\ep_0)}{N}(m-1)<1,
\end{equation*}
which concludes the proof.
\end{proof}

Now we shall show the following estimate of $I_1$.

\begin{lem}\label{I1esti}
Let $N\geq3$, $R>0$, $m\geq1$, $M_0>0$ and $L>0$, 
and assume that $m$ and $\chi$ satisfy \eqref{chi} and \eqref{chi4} with \eqref{lem36}.
Then there exist $C>0$ and $\ep_0>0$ such that the following property holds\,$:$

If $u_0$ satisfies \eqref{ini} and \eqref{iniprop} with $p(\ep_0)>0$ and $L>0$, 
as well as $\io u_0=M_0$, then for all
\begin{equation*}
\gamma\in\left(1-\dfrac{2}{N}-\dfrac{p(\ep_0)}{N}(m-1), \,\min\left\{2-\dfrac{4}{N}-\dfrac{2p(\ep_0)}{N}(m-1)-\left(1-\dfrac{2}{N}\right)k,\,1\right\}\right),
\end{equation*}
and for each $s_0\in(0,R^N)$ and $t\in(0,\tmax)$,
\begin{align*}
I_1(s_0,t)\geq&\,-Cs_0^{\frac{3}{2}-\frac{2}{N}-\gamma-\frac{p(\ep)}{N}(m-1)+\frac{\alpha}{2}}\left\{\int_0^{s_0}s^{-\alpha}(s_0-s)w(s,t)w_s(s,t)\,ds\right\}^\frac{1}{2}-Cs_0^{3-\frac{2}{N}-\gamma},
\end{align*}
where $\alpha=\alpha(\gamma)\in(0,1)$ is defined in \eqref{alphadef}.
\end{lem}

\begin{proof}
Firstly, from Lemma \ref{gammaexist}, there exists a constant $\ep_0>0$ such that
\begin{equation*}
J:=\left(1-\dfrac{2}{N}-\dfrac{p(\ep_0)}{N}(m-1), \,\min\left\{2-\dfrac{4}{N}-\dfrac{2p(\ep_0)}{N}(m-1)-\left(1-\dfrac{2}{N}\right)k,\,1\right\}\right)\neq\emptyset
\end{equation*}
is a subset an interval $(0,1)$. 
Therefore, we can pick $\gamma\in J\subset(0,1)$.
Thus from an integration by parts and the fact $\gamma<1<2-\frac{2}{N}$, we obtain
\begin{align}\label{I1-1}
I_1(s_0,t)=&\,N\int_0^{s_0}s^{2-\frac{2}{N}-\gamma}(s_0-s)\left\{(Nw_s(s,t)+1)^m\right\}_s\,ds\\\notag
          =&-N\int_0^{s_0}\left\{s^{2-\frac{2}{N}-\gamma}(s_0-s)\right\}_s(Nw_s(s,t)+1)^m\,ds\\\notag
           &+\left[Ns^{2-\frac{2}{N}-\gamma}(s_0-s)(Nw_s(s,t)+1)^m\right]_0^{s_0}\\\notag
          =&-N\left(2-\dfrac{2}{N}-\gamma\right)\int_0^{s_0}s^{1-\frac{2}{N}-\gamma}(s_0-s)(Nw_s(s,t)+1)^m\,ds\\\notag
           &+\int_0^{s_0}s^{2-\frac{2}{N}-\gamma}(Nw_s(s,t)+1)^m\,ds\\\notag
          \geq&-N\left(2-\dfrac{2}{N}-\gamma\right)\int_0^{s_0}s^{1-\frac{2}{N}-\gamma}(s_0-s)(Nw_s(s,t)+1)^m\,ds
\end{align}
for all $s_0\in(0,R^N)$ and $t\in(0,\tmax)$. 
On the other hand, in view of \eqref{iniprop}, \eqref{westi} and Lemma \ref{uesti},
there exist $K>0$ such that
\begin{equation*}
Nw_s(s,t)=u(s^\frac{1}{N},t)\leq Ks^{-\frac{p(\ep_0)}{N}}
\end{equation*}
for all $s\in(0,R^N)$ and $t\in(0,\tmax)$.
Therefore, from the fact $m\geq1$ and the inequality $(A+B)^m\leq2^{m-1}(A^m+B^m)$\,$(A, B\geq0)$, 
we infer for all $s\in(0,R^N)$ and $t\in(0,\tmax)$,
\begin{align*}
(Nw_s(s,t)+1)^m&\leq2^{m-1}\left((Nw_s(s,t))^m+1^m\right)\\
               &=2^{m-1}\left((Nw_s(s,t))^{m-1}\cdot Nw_s(s,t)+1\right)\\
               &\leq2^{m-1}K^{m-1}s^{-\frac{p(\ep_0)}{N}(m-1)}Nw_s(s,t)+2^{m-1}\\
               &\leq C_1s^{-\frac{p(\ep_0)}{N}(m-1)}Nw_s(s,t)+C_1,
\end{align*}
where $C_1:=\max\{2^{m-1}K^{m-1}N,2^{m-1}\}>0$. 
Applying this to \eqref{I1-1}, we deduce that
\begin{align}\label{I1-2}
I_1(s_0,t)\geq&\,-NC_1\left(2-\dfrac{2}{N}-\gamma\right)\int_0^{s_0}s^{1-\frac{2}{N}-\gamma-\frac{p(\ep_0)}{N}(m-1)}(s_0-s)w_s(s,t)\,ds\\\notag
&-NC_1\left(2-\dfrac{2}{N}-\gamma\right)\int_0^{s_0}s^{1-\frac{2}{N}-\gamma}(s_0-s)\,ds
\end{align}
for all $s_0\in(0,R^N)$ and $t\in(0,\tmax)$.
Next, noting that $\gamma>1-\frac{2}{N}-\frac{p(\ep_0)}{N}(m-1)$ and that $s_0-s\leq s_0$ for $s\in(0,s_0)$, 
and using an integration by parts, we have
\begin{align*}
&\int_0^{s_0}s^{1-\frac{2}{N}-\gamma-\frac{p(\ep_0)}{N}(m-1)}(s_0-s)w_s(s,t)\,ds\\
&=-\int_0^{s_0}\left\{s^{1-\frac{2}{N}-\gamma-\frac{p(\ep_0)}{N}(m-1)}(s_0-s)\right\}_sw(s,t)\,ds\\
&\quad+\left[s^{1-\frac{2}{N}-\gamma-\frac{p(\ep_0)}{N}(m-1)}(s_0-s)w(s,t)\right]_0^{s_0}\\
&=-\left(1-\dfrac{2}{N}-\gamma-\dfrac{p(\ep_0)}{N}(m-1)\right)\int_0^{s_0}s^{-\frac{2}{N}-\gamma-\frac{p(\ep_0)}{N}(m-1)}(s_0-s)w(s,t)\,ds\\
&\quad+\int_0^{s_0}s^{1-\frac{2}{N}-\gamma-\frac{p(\ep_0)}{N}(m-1)}w(s,t)\,ds\\
&\leq\left(\gamma-1+\dfrac{2}{N}+\dfrac{p(\ep_0)}{N}(m-1)\right)s_0\int_0^{s_0}s^{-\frac{2}{N}-\gamma-\frac{p(\ep_0)}{N}(m-1)}w(s,t)\,ds\\
&\quad+s_0\int_0^{s_0}s^{-\frac{2}{N}-\gamma-\frac{p(\ep_0)}{N}(m-1)}w(s,t)\,ds\\
&=\left(\gamma+\dfrac{2}{N}+\dfrac{p(\ep_0)}{N}(m-1)\right)s_0\int_0^{s_0}s^{-\frac{2}{N}-\gamma-\frac{p(\ep_0)}{N}(m-1)}w(s,t)\,ds
\end{align*}
for all $s_0\in(0,R^N)$ and $t\in(0,\tmax)$. 
Furthermore, from $\gamma<2-\frac{2}{N}$ we can find that 
\begin{align*}
\left(2-\dfrac{2}{N}-\gamma\right)\int_0^{s_0}s^{1-\frac{2}{N}-\gamma}(s_0-s)\,ds
&\leq \left(2-\dfrac{2}{N}-\gamma\right)s_0\int_0^{s_0}s^{1-\frac{2}{N}-\gamma}\,ds\\\notag
&= s_0^{3-\frac{2}{N}-\gamma}
\end{align*}
for all $s_0\in(0,R^N)$ and $t\in(0,\tmax)$.
Hence, plugging these estimates into \eqref{I1-2}, we see that
\begin{align}\label{I1-3}
&I_1(s_0,t)\\\notag
&\geq-NC_1\left(2-\dfrac{2}{N}-\gamma\right)\left(\gamma+\dfrac{2}{N}+\dfrac{p(\ep_0)}{N}(m-1)\right)s_0\int_0^{s_0}s^{-\frac{2}{N}-\gamma-\frac{p(\ep_0)}{N}(m-1)}w(s,t)\,ds\\\notag
&\quad-NC_1s_0^{3-\frac{2}{N}-\gamma}
\end{align}
for all $s_0\in(0,R^N)$ and $t\in(0,\tmax)$.
Next, by definition of $\alpha$ and the condition 
$\gamma<2-\frac{4}{N}-\frac{2p(\ep_0)}{N}(m-1)-\left(1-\frac{2}{N}\right)k$, 
we can give the estimate 
\begin{align*}
&1-\dfrac{2}{N}-\gamma-\dfrac{p(\ep_0)}{N}(m-1)+\dfrac{\alpha}{2}\\
&=1-\dfrac{2}{N}-\dfrac{\gamma}{2}-\dfrac{p(\ep_0)}{N}(m-1)-\dfrac{1}{2}\left(1-\dfrac{2}{N}\right)k\\
&=\dfrac{1}{2}\left\{2-\dfrac{4}{N}-\dfrac{2p(\ep_0)}{N}(m-1)-\left(1-\dfrac{2}{N}\right)k-\gamma\right\}>0.
\end{align*}
Thus, from this estimate and Lemmas \ref{lem35} and \eqref{betafunc}, we obtain for all $s_0\in(0,R^N)$ and $t\in(0,\tmax)$,
\begin{align*}
&s_0\int_0^{s_0}s^{-\frac{2}{N}-\gamma-\frac{p(\ep_0)}{N}(m-1)}w(s,t)\,ds\\
&\leq\sqrt{2}s_0\int_0^{s_0}s^{-\frac{2}{N}-\gamma-\frac{p(\ep_0)}{N}(m-1)+\frac{\alpha}{2}}(s_0-s)^{-\frac{1}{2}}\left\{\int_0^{s_0}s^{-\alpha}(s_0-\sigma)w(\sigma,t)w_s(\sigma,t)\,d\sigma\right\}^\frac{1}{2}\,ds\\
&=\sqrt{2}B\left(1-\dfrac{2}{N}-\gamma-\dfrac{p(\ep_0)}{N}(m-1)+\dfrac{\alpha}{2},\,\dfrac{1}{2}\right)
s_0^{\frac{3}{2}-\frac{2}{N}-\gamma-\frac{p(\ep_0)}{N}(m-1)}\\
&\quad\times\left\{\int_0^{s_0}s^{-\alpha}(s_0-\sigma)w(\sigma,t)w_s(\sigma,t)\,d\sigma\right\}^\frac{1}{2}.
\end{align*}
Finally, applying the above inequality to \eqref{I1-3}, 
we can pick $C_2>0$ such that 
\begin{equation*}
I_1(s_0,t)\geq-C_2s_0^{\frac{3}{2}-\frac{2}{N}-\gamma-\frac{p(\ep_0)}{N}(m-1)+\frac{\alpha}{2}}\left\{\int_0^{s_0}s^{-\alpha}(s_0-s)w(s,t)w_s(s,t)\,ds\right\}^\frac{1}{2}-NC_1s_0^{3-\frac{2}{N}-\gamma}
\end{equation*}
for all $s_0\in(0,R^N)$ and $t\in(0,\tmax)$, which concludes the proof.
\end{proof}

The next step gives the estimate for $I_3$.

\begin{lem}\label{betafunc2}
Let $a\in(1,2)$ and $b\in(0,1)$. Then there exists $C=C(a,b)>0$ such that 
if $s_0>0$, then 
\begin{equation*}
\int_0^{s_0}\int_\sigma^{s_0}\xi^{-a}(s_0-\xi)^{-b}\,d\xi d\sigma\leq Cs_0^{-b}s^{2-a}
\end{equation*}
for all $s\in(0,s_0)$.
\end{lem}

\begin{lem}\label{I3esti}
Let $N\geq3$, $R>0$, $m\geq 1$ and $M_0>0$. 
Assume that $m$ and $\chi$ satisfy \eqref{chi} and \eqref{chi4} with \eqref{lem36}, 
and define $\ep_0>0$ same as Lemma \ref{I1esti}.
Then there exists $C>0$ such that the following property\,$:$

If $u_0$ satisfies \eqref{ini} and $\io u_0=M_0$, 
then for all $\gamma\in(0,1)$ with \eqref{propgamma}, 
\begin{align*}
I_3(s_0,t)\geq&\,-Cs_0^{\frac{1}{2}+\frac{2}{N}-\gamma+\frac{\alpha}{2}}\left\{\int_0^{s_0}s^{-\alpha}(s_0-s)w(s,t)w_s(s,t)\,ds\right\}^\frac{1}{2}\\
&-Cs_0^{-\gamma+\frac{2}{N}+\alpha}\int_0^{s_0}s^{-\alpha}(s_0-s)w(s,t)w_s(s,t)\,ds
\end{align*}
for all $s_0\in(0,R^N)$ and $t\in(0,\tmax)$.
\end{lem}

\begin{proof}
At first, we fix $u_0$ with \eqref{ini} and $\io u_0=M_0$. 
Since Lemma \ref{regu} and \eqref{zesti}, we can find the constant $\eta>0$ defined by \eqref{etadef} 
such that the following lower pointwise estimate holds:
\begin{equation*}
Nz_s(s,t)=v(s^\frac{1}{N},t)\geq\eta \quad\mbox{for all}\ s\in(0,R^N)\ \ \mbox{and}\ \ t\in(0,\tmax).
\end{equation*}
Thus from \eqref{chi3} and the nonnegativity of $w_s$ and $z$, 
we infer for all $s_0\in(0,R^N)$ and $t\in(0,\tmax)$,
\begin{align*}
I_3(s_0,t)&\geq-N\|\chi\|_{L^{\infty}(\eta,\infty)}\int_0^{s_0}s^{-\gamma}(s_0-s)w_s(s,t)z(s,t)\,ds.
\end{align*}
Moreover, by an integuration by parts, the fact $w(0,t)=0$ and $z(0,t)=0$ for all $t\in(0,\tmax)$ and 
the nonnegativity of $z_s$, we obtain
\begin{align*}
&-\int_0^{s_0}s^{-\gamma}(s_0-s)w_s(s,t)z(s,t)\,ds\\
&\geq\int_0^{s_0}\left\{s^{-\gamma}(s_0-s)z(s,t)\right\}_sw(s,t)\,ds
-\left[s^{-\gamma}(s_0-s)z(s,t)w(s,t)\right]_0^{s_0}\\\notag
&=-\gamma\int_0^{s_0}s^{-\gamma-1}(s_0-s)z(s,t)w(s,t)\,ds
-\int_0^{s_0}s^{-\gamma}z(s,t)w(s,t)\,ds\\\notag
&\quad+\int_0^{s_0}s^{-\gamma}(s_0-s)z_s(s,t)w(s,t)\,ds\\\notag
&\geq-\gamma\int_0^{s_0}s^{-\gamma-1}(s_0-s)z(s,t)w(s,t)\,ds
-\int_0^{s_0}s^{-\gamma}z(s,t)w(s,t)\,ds\\\notag
\end{align*}
for all $s_0\in(0,R^N)$ and $t\in(0,\tmax)$. 
Using the nonnegativity of $z$ and $w$, and inequalities $s_0-s\leq s_0$ and $s^{-\gamma}\leq s_0s^{-\gamma-1}$ for all $s\in(0,s_0)$, 
we see that
\begin{align}\label{I3-1}
I_3(s_0,t)
&\geq-\gamma N\|\chi\|_{L^{\infty}(\eta,\infty)}s_0\int_0^{s_0}s^{-\gamma-1}z(s,t)w(s,t)\,ds\\\notag
&\quad-N\|\chi\|_{L^{\infty}(\eta,\infty)}s_0\int_0^{s_0}s^{-\gamma-1}z(s,t)w(s,t)\,ds\\\notag
&=-(\gamma+1)N\|\chi\|_{L^{\infty}(\eta,\infty)}s_0\int_0^{s_0}s^{-\gamma-1}z(s,t)w(s,t)\,ds
\end{align}
for all $s_0\in(0,R^N)$ and $t\in(0,\tmax)$. 
Here, by an argument similar to the proof of \cite[Lemma 4.7]{wink3} we can see that 
there exists $C_1=C_1(R,M_0)>0$ such that
\begin{equation}\label{zesti3}
z(s,t)\leq\dfrac{C_1}{N}s_0^{\frac{2}{N}-1}s+\dfrac{1}{N^2}\int_0^s\int_\sigma^{s_0}\xi^{\frac{2}{N}-2}w(\xi,t)\,d\xi d\sigma
\end{equation}
for all $s\in(0,s_0)$ and $t\in(0,\tmax)$. 
On the other hand, since $N\geq3$, \eqref{lem36}, \eqref{propgamma} and 
$\alpha=\alpha(\gamma)\in(0,1)$, we have the innequality
\begin{align*}
-\dfrac{2}{N}+2-\dfrac{\alpha}{2}-1
&=\dfrac{1}{2}\left(\left(1-\dfrac{2}{N}\right)(k+2)-\gamma\right)\\
&>\dfrac{1}{2}\left(\dfrac{4}{3}-\gamma\right)>0
\end{align*}
as well as
\begin{align*}
2-\left(-\dfrac{2}{N}+2-\dfrac{\alpha}{2}\right)
=\dfrac{2}{N}+\dfrac{\alpha}{2}>0,
\end{align*}
that is to say, we obtain
\begin{equation}\label{gammaesti4}
-\dfrac{2}{N}+2-\dfrac{\alpha}{2}\in(1,2).
\end{equation}
Therefore, using Lemmas \ref{betafunc2} and \ref{lem35} and \eqref{zesti3}, 
there exists some constant $C_2>0$ such that
\begin{align}\label{zesti4}
&\int_0^s\int_\sigma^{s_0}\xi^{\frac{2}{N}-2}w(\xi,t)\,d\xi d\sigma\\\notag
&\leq\sqrt{2}\int_0^s\int_\sigma^{s_0}\xi^{\frac{2}{N}-2+\frac{\alpha}{2}}(s_0-\xi)^{-\frac{1}{2}}\,d\xi d\sigma
\cdot\left\{\int_0^{s_0}s^{-\alpha}(s_0-s)w(s,t)w_s(s,t)\,ds\right\}^\frac{1}{2}\\\notag
&\leq\sqrt{2}C_2s_0^{-\frac{1}{2}}s^{\frac{2}{N}+\frac{\alpha}{2}}
\left\{\int_0^{s_0}s^{-\alpha}(s_0-s)w(s,t)w_s(s,t)\,ds\right\}^\frac{1}{2}
\end{align}
for all $s\in(0,s_0)$ and $t\in(0,\tmax)$.
Thus plugging \eqref{zesti3} and \eqref{zesti4} into \eqref{I3-1} we have 
\begin{align}\label{I3-2}
I_3(s_0,t)
&\geq-C_1(\gamma+1)N\|\chi\|_{L^{\infty}(\eta,\infty)}s_0^\frac{2}{N}\int_0^{s_0}s^{-\gamma}w(s,t)\,ds\\\notag
&\quad-\sqrt{2}C_2(\gamma+1)N\|\chi\|_{L^{\infty}(\eta,\infty)}s_0^\frac{1}{2}\int_0^{s_0}s^{-\gamma-1+\frac{2}{N}+\frac{\alpha}{2}}w(s,t)\,ds\\\notag
&\quad\times\left\{\int_0^{s_0}s^{-\alpha}(s_0-s)w(s,t)w_s(s,t)\,ds\right\}^\frac{1}{2}
\end{align}
for all $s_0\in(0,R^N)$ and $t\in(0,\tmax)$. 
Next, using \eqref{gammaesti4} and $k\in(0,\frac{2}{N-2})$ we can find that 
\begin{align*}
1-\gamma+\dfrac{\alpha}{2}-\left(-\gamma+\dfrac{2}{N}+\alpha\right)
&=1-\dfrac{2}{N}-\dfrac{\alpha}{2}\\
&=-\dfrac{2}{N}+2-\dfrac{\alpha}{2}-1>0,
\end{align*}
as well as 
\begin{equation*}
-\gamma+\dfrac{2}{N}+\alpha=\dfrac{2}{N}-\left(1-\dfrac{2}{N}\right)k
=\dfrac{N-2}{N}\left(\dfrac{2}{N-2}-k\right)>0.
\end{equation*}
Therefore, noting that Lemma \ref{lem35} we can find that
\begin{align*}
s_0^\frac{2}{N}\int_0^{s_0}s^{-\gamma}w(s,t)\,ds
&\leq\sqrt{2}s_0^\frac{2}{N}\int_0^{s_0}s^{-\gamma+\frac{\alpha}{2}}(s_0-s)^{-\frac{1}{2}}\,ds\\
&\quad\times\left\{\int_0^{s_0}s^{-\alpha}(s_0-s)w(s,t)w_s(s,t)\,ds\right\}^\frac{1}{2}\\
&=\sqrt{2}B\left(1-\gamma+\dfrac{\alpha}{2},\,\dfrac{1}{2}\right)s_0^{\frac{1}{2}+\frac{2}{N}-\gamma+\frac{\alpha}{2}}\\
&\quad\times\left\{\int_0^{s_0}s^{-\alpha}(s_0-s)w(s,t)w_s(s,t)\,ds\right\}^\frac{1}{2}, 
\end{align*}
for all $s\in(0,R^N)$ and $t\in(0,\tmax)$, as well as
\begin{align*}
s_0^\frac{1}{2}\int_0^{s_0}s^{-\gamma+\frac{2}{N}+\frac{\alpha}{2}-1}w(s,t)\,ds
&\leq\sqrt{2}s_0^\frac{1}{2}\int_0^{s_0}s^{-\gamma+\frac{2}{N}+\alpha-1}(s_0-s)^{-\frac{1}{2}}\,ds\\
&\quad\times\left\{\int_0^{s_0}s^{-\alpha}(s_0-s)w(s,t)w_s(s,t)\,ds\right\}^\frac{1}{2}\\
&=\sqrt{2}B\left(-\gamma+\dfrac{2}{N}+\alpha,\,\dfrac{1}{2}\right)s_0^{-\gamma+\frac{2}{N}+\alpha}\\
&\quad\times\left\{\int_0^{s_0}s^{-\alpha}(s_0-s)w(s,t)w_s(s,t)\,ds\right\}^\frac{1}{2}, 
\end{align*}
for all $s\in(0,R^N)$ and $t\in(0,\tmax)$.
Finaliy applying these inequalities to \eqref{I3-2}, we can prove the conclusion.
\end{proof}

\begin{proof}[\bf Proof of Proposition \ref{ODI}]
At first, we pick $\ep_0>0$ as in Lemma \ref{I1esti}.
Commbining Lemmas \ref{lem3.2}, \ref{I2esti}, \ref{I1esti} and \ref{I3esti}, 
there exist $C_1>0$, $C_2>0$ and $C_3>0$ such that 
\begin{align*}
\dfrac{\pa\phi}{\pa t}(s_0,t)
&\geq-C_1s_0^{\frac{3}{2}-\frac{2}{N}-\gamma-\frac{p(\ep_0)}{N}(m-1)+\frac{\alpha}{2}}
\left\{\int_0^{s_0}s^{-\alpha}(s_0-s)w(s,t)w_s(s,t)\,ds\right\}^\frac{1}{2}\\
&\quad-C_1s_0^{3-\frac{2}{N}-\gamma}\\
&\quad+C_2\int_0^{s_0}s^{-\alpha}(s_0-s)w(s,t)w_s(s,t)\,ds\\
&\quad-C_3s_0^{\frac{1}{2}+\frac{2}{N}-\gamma+\frac{\alpha}{2}}\left\{\int_0^{s_0}s^{-\alpha}(s_0-s)w(s,t)w_s(s,t)\,ds\right\}^\frac{1}{2}\\
&\quad-C_3s_0^{-\gamma+\frac{2}{N}+\alpha}\int_0^{s_0}s^{-\alpha}(s_0-s)w(s,t)w_s(s,t)\,ds
\end{align*}
for all $s_0\in(0,R^N)$ and $t\in(0,\tmax)$. 
Using Young's inequality for all $\delta>0$, 
we can find $C_5=C_5(\delta)>0$ such that
\begin{align*}
\dfrac{\pa\phi}{\pa t}(s_0,t)
&\geq C_2\int_0^{s_0}s^{-\alpha}(s_0-s)w(s,t)w_s(s,t)\,ds\\
&\quad-\delta\int_0^{s_0}s^{-\alpha}(s_0-s)w(s,t)w_s(s,t)\,ds\\
&\quad-C_3s_0^{-\gamma+\frac{2}{N}+\alpha}\int_0^{s_0}s^{-\alpha}(s_0-s)w(s,t)w_s(s,t)\,ds\\
&\quad-C_5\left(s_0^{3-\frac{4}{N}-2\gamma-\frac{2p(\ep_0)}{N}(m-1)+\alpha}
+s_0^{1+\frac{4}{N}-2\gamma+\alpha}+s_0^{3-\frac{2}{N}-\gamma}\right)
\end{align*}
for all $s_0\in(0,R^N)$ and $t\in(0,\tmax)$. 
By the definition $\alpha$ and the fact $k\in(0,\frac{2}{N-2})$ we obtain
\begin{equation*}
-\gamma+\dfrac{2}{N}+\alpha=\dfrac{1}{N}(2-(N-2)k)>0,
\end{equation*}
moreover, we pick 
\begin{equation*}
0<s_0<s_1:=\left(\dfrac{C_2}{4C_3}\right)^\frac{N}{2-(N-2)k}
\end{equation*}
and
\begin{equation*}
\delta=\dfrac{C_2}{4}>0.
\end{equation*}
Then we can estimate 
\begin{align}\label{phi2}
\dfrac{\pa\phi}{\pa t}(s_0,t)
&\geq \dfrac{C_2}{2}\int_0^{s_0}s^{-\alpha}(s_0-s)w(s,t)w_s(s,t)\,ds\\\notag
&\quad-C_5\left(s_0^{3-\frac{4}{N}-2\gamma-\frac{2p(\ep_0)}{N}(m-1)+\alpha}
+s_0^{1+\frac{4}{N}-2\gamma+\alpha}+s_0^{3-\frac{2}{N}-\gamma}\right)
\end{align}
for all $s_0\in(0,s_1)$ and $t\in(0,\tmax)$.
Next, putting
\begin{equation*}
\theta_1:=\max\left\{\dfrac{4}{N}+\dfrac{2p(\ep_0)}{N}(m-1)+\left(1-\dfrac{2}{N}\right)k,\,
2-\dfrac{4}{N}+\left(1-\dfrac{2}{N}\right)k,\,
\dfrac{2}{N}\right\},
\end{equation*}
we have $\theta_1\in(0,2-(1-\frac{2}{N})k)$, that is, we obtain 
\begin{align*}
&2-\left(1-\dfrac{2}{N}\right)k-\left(\dfrac{4}{N}+\dfrac{2p(\ep_0)}{N}(m-1)+\left(1-\dfrac{2}{N}\right)k\right)\\
&=2\left(1-\dfrac{2}{N}-\dfrac{p(\ep_0)}{N}(m-1)-\left(1-\dfrac{2}{N}\right)k\right)>0
\end{align*}
and
\begin{align*}
&2-\left(1-\dfrac{2}{N}\right)k-\left(2-\dfrac{4}{N}+\left(1-\dfrac{2}{N}\right)k\right)=\dfrac{4}{N}>0
\end{align*}
as well as
\begin{align*}
2-\left(1-\dfrac{2}{N}\right)k-\dfrac{2}{N}>2-\dfrac{N-2}{N}\cdot\dfrac{2}{N-2}-\dfrac{2}{N}=2-\dfrac{4}{N}>0
\end{align*}
from the definition of $\ep_0>0$ and the condition $k<\frac{2}{N-2}$.
On the other hand, noting that $\gamma<2-(1-\frac{2}{N})k$ 
and that $s_0-s\leq s_0$ for all $s\in(0,s_0)$, 
we see that Lemma \ref{lem35} implies
\begin{align*}
\phi(s_0,t)
&=\int_0^{s_0}s^{-\gamma}(s_0-s)w(s,t)\,ds\\
&\leq s_0\int_0^{s_0}s^{-\gamma}w(s,t)\,ds\\
&\leq\sqrt{2}s_0\int_0^{s_0}s^{-\gamma+\alpha}(s_0-s)^{-\frac{1}{2}}\,ds
\cdot\left\{\int_0^{s_0}s^{-\alpha}(s_0-s)w(s,t)w_s(s,t)\,ds\right\}^\frac{1}{2}\\
&=\sqrt{2}B\left(1-\dfrac{\gamma}{2}-\dfrac{1}{2}\left(1-\dfrac{2}{N}\right)k,\,\dfrac{1}{2}\right)s_0^{\frac{3}{2}-\frac{\gamma}{2}-\frac{1}{2}(1-\frac{2}{N})k}\\
&\quad\times\left\{\int_0^{s_0}s^{-\alpha}(s_0-s)w(s,t)w_s(s,t)\,ds\right\}^\frac{1}{2}.
\end{align*}
Invoking $\gamma\in(0,1)$, $\theta_1\in(0,2-(1-\frac{2}{N})k)$ and $s_0<R^N$, 
we can deduce from \eqref{phi2} that there exists $C_4>0$ such that 
\begin{equation*}
\dfrac{\pa\phi}{\pa t}(s_0,t)\geq C_4 s_0^{-3+\gamma+(1-\frac{2}{N})k}\phi^2(s_0,t)-C_4s_0^{3-\gamma-\theta_1}
\end{equation*}
for all $s_0\in(0,s_1)$ and $t\in(0,\tmax)$, which concludes the proof.
\end{proof}

\section{Proof of the main theorem}\label{4}

Now we prove the main theorem. 
By the similar argument as that in \cite[Lemma 4.1]{BFL} we can prove the following lemma 
which is needed for the proof of Theorem \ref{thm}.

\begin{lem}\label{lem41}
Let $\gamma\in(0,1)$, $s_0\in(0,R^N)$, $M_1>0$ and $\eta\in(0,1)$ 
and set $s_\eta:=(1-\eta)s_0$ as well as $r_1:=s_\eta^\frac{1}{N}.$ 
If
\begin{align*}
\int_{B_{r_1}(0)}u_0\geq M_1,
\end{align*}
then
\begin{align*}
\phi(s_0,0)\geq \dfrac{\eta^2M_1}{\omega_{N-1}}\cdot s_0^{2-\gamma}.
\end{align*}
\end{lem}
\begin{proof}
We use positivity and monotonicity of $w_0:=w(\cdot,0)$, $w_0(s_\eta):=\frac{1}{\omega_{N-1}}\int_{B_{r_1}}u_0$ 
and $s_0-s_1=\eta s_0$ as well as the fuct that $1-(1-\eta)^{1-\gamma}\geq \inf_{\xi\in(0,\eta)}(1-\gamma)(1-\xi)^{-\gamma}\eta=(1-\gamma)\eta$
holds by the mean value theorem, to see that
\begin{align*}
\phi(s_0,0)&\geq\int_0^{s_0}s^{-\gamma}(s_0-s)w_0(s)\,ds\\
&\geq w_0(s_\eta)\int^{s_0}_{s_1}s^{-\gamma}(s_0-s_\eta)\,ds\\
&\geq \dfrac{\eta M_1}{(1-\gamma)\omega_{N-1}}s_0(s_0^{1-\gamma}-s_\eta^{1-\gamma})\\
&\geq \dfrac{\eta^2 M_1}{\omega_{N-1}}\cdot s_0^{2-\gamma}.
\end{align*}
\end{proof}

\begin{proof}[\bf Proof of Theorem \ref{thm}]
From Proposition \ref{ODI} we see that \eqref{propmain} holds. 
Using Lemma \ref{lem41} and an argument similar to that in the proof of \cite[Theorem 1.1]{BFL} 
we can find that $\tmax<T<\infty$. Thanks to Lemma \ref{regu}, 
we arrive at the conclusion \eqref{ulim}.
\end{proof}

\end{document}